\documentclass[journal]{IEEEtran}

\usepackage{amsmath,amssymb,amsfonts}
\usepackage{mathtools}
\usepackage{booktabs}
\usepackage{array}
\usepackage{cite, xcolor}

\newcolumntype{L}[1]{%
  >{\raggedright\arraybackslash}m{#1}%
}
\newcommand{\supp}{\operatorname{supp}}
\newcommand{\N}{\mathcal N}
\newcommand{\Sn}{\mathcal S_n}
\newcommand{\Tzero}{\mathcal T_0}
\newcommand{\TB}{\mathcal T_{\mathrm B}}
\newcommand{\TAS}{\mathcal T_{\mathrm{AS}}}
\newcommand{\TNE}{\mathcal T_{\mathrm{NE}}}
\newcommand{\U}{\mathcal U}
\newcommand{\E}{\mathcal E}
\newcommand{\one}{\mathbf 1}
\newcommand{\ellv}{\operatorname{lev}_{\mathrm{AS}}}
\newcommand{\ellne}{\operatorname{lev}_{\mathrm{NE}}}
\newcommand{\hscr}{h_{\mathrm{scr}}}

\newtheorem{definition}{Definition}
\newtheorem{lemma}{Lemma}
\newtheorem{theorem}{Theorem}
\newtheorem{proposition}{Proposition}

\newtheorem{remark}{Remark}
\newtheorem{example}{Example}

\begin{document}

\title{Sarymsakov-Type Semigroups With Support Dominance and Nested Mixing Cores}

\author{Shun-Pin Hsu,~\IEEEmembership{Member,~IEEE}%
\thanks{The author is with the Department of Electrical Engineering, National Chung Hsing University, Taichung 402, Taiwan (e-mail: shsu@nchu.edu.tw).}%
\thanks{This work was supported in part by the National Science and Technology Council, Taiwan, under Grant NSTC 115-2221-E-005-087.}}

\maketitle

\begin{abstract}
This note develops three multiplication-closed families of stochastic matrices for consensus under arbitrary switching. The first is a support-dominance enlargement of an earlier power-generated construction, in which exact support-pattern matching is relaxed to a dominance-type condition without losing closure under multiplication. The second is a fixed-core class with an exact uniform scrambling horizon and corresponding convergence-rate guarantees for switching products. The third and main contribution is the almost Sarymsakov class, obtained by allowing variable mixing-core dimensions under suitable structural conditions. This new semigroup class properly contains the classical Sarymsakov class and admits a dimension-only scrambling-horizon bound. Examples illustrate the necessity of the proposed structural assumptions and clarify the sharpness or limitations of the derived bounds.
\end{abstract}


\section{Introduction}
Products of stochastic matrices are central to nonhomogeneous Markov chains, distributed coordination, learning over networks, and switched consensus. Fix $n\ge2$, let $\Sn$ be the set of $n\times n$ row-stochastic matrices, and consider
\begin{equation}
 x(k+1)=P(k)x(k),
 \label{eq:system}
\end{equation}
where $x(k)\in\mathbb R^n$ is the column state and $P(k)\in\Sn$. Consensus means $x(k)\to c\one$ for every initial state and is governed by the backward products $P(k-1)\cdots P(0)$. A fundamental objective is therefore to identify multiplication-closed matrix classes whose compact subfamilies generate convergent products under arbitrary switching. When such a class has a finite uniform scrambling horizon, it also yields blockwise contraction estimates for the consensus disagreement.

The classical theory of weak ergodicity and products of indecomposable, aperiodic stochastic matrices was developed in \cite{hajnal1958,sarymsakov1961,wolfowitz1963,seneta1981,hartfielseneta1990,hartfiel2002}. Sarymsakov matrices provide a particularly useful support-defined semigroup, and every compact subset of the class is a consensus set; related graph-composition and asynchronous results appear in \cite{cao2008graph,xia2014}. Nevertheless, the Sarymsakov class is a proper subset of the stochastic, indecomposable, and aperiodic (SIA) matrices. Enlarging it while retaining multiplication closure is delicate: a uniform SIA-index bound is not preserved in general, and existing constructive extensions are tied to prescribed zero--nonzero patterns or finitely many power patterns \cite{xia2019,hsu2023}. This motivates two questions: can exact support equality be relaxed to support inclusion, and can reference-pattern-independent semigroups accommodate factors with different effective mixing dimensions?

Recent studies have examined complementary aspects of stochastic-matrix products, including ergodicity coefficients, convergence rates, complexity of short convergent products, and finite-block contraction \cite{depasquale2024,xia2025,chevalier2017,ofir2025,ofirmorse2026}. The present paper addresses a different construction problem: it seeks support-defined semigroups whose admissibility conditions are preserved by multiplication and whose finite scrambling horizons follow from explicit structural assumptions.

The principal contribution of the note is the proposal of a new stochastic-matrix semigroup, termed the \emph{almost Sarymsakov class}. It allows the effective ordered mixing-core dimension to vary among factors while preserving multiplication closure; row-allowable access then ensures SIA behavior and consensus. We prove level monotonicity, proper containment of the classical Sarymsakov class for $n\ge3$, and a dimension-only scrambling-horizon bound. Two supporting results build the theory: a support-dominance upper hull that enlarges the equal-pattern method of \cite{hsu2023}, and a fixed order-$r$ core class  with the exact horizon $\min\{r,n-1\}$ and compact-family convergence rates. Examples establish strict inclusions, low-dimensional sharpness, and the necessity of access, higher-order compatibility, and a common ordering.

The remainder of the note is organized as follows. Section~II gives the support-based preliminaries. Section~III develops the power-generated class. Sections~IV and V present the fixed-core and variable-core constructions, Section~VI compares the classes, and Section~VII concludes the note.

\section{Preliminaries}
Let $\N:=\{1,\ldots,n\}$. The directed graph of a nonnegative matrix $A=[a_{ij}]$ has the arc $i\to j$ whenever $a_{ij}>0$. A nonnegative matrix is \emph{row-allowable} if every row has a positive entry. For square $A$ and $I\subseteq\N$, define the \emph{consequent set} of $I$ under $A$ by
\begin{equation}
 F_A(I):=\{j\in\N:a_{ij}>0\text{ for some }i\in I\}.
 \label{eq:F}
\end{equation}
For compatible nonnegative matrices,
\begin{equation}
 F_{AB}(I)=F_B(F_A(I)).
 \label{eq:composition}
\end{equation}
Write $A\succeq B$ when $\supp(B)\subseteq\supp(A)$. Thus $A\succeq B$ means that every arc of $B$ is retained in $A$, although the positive weights may be different, and it implies $F_A(I)\supseteq F_B(I)$.

A row-allowable square nonnegative matrix is \emph{scrambling} if every pair of its rows has a common positive column. This is a support property and is therefore meaningful even when the matrix is not stochastic. A stochastic matrix $P$ is SIA if $P^k\to\one v^{\mathsf T}$ for some probability vector $v$. A row-allowable square nonnegative matrix $A$ is \emph{GE} if, for every pair of disjoint nonempty sets $I,J$, either $F_A(I)\cap F_A(J)\ne\varnothing$, or
\begin{equation}
 |F_A(I)\cup F_A(J)|\ge |I\cup J|.
 \label{eq:GE}
\end{equation}
It is \emph{GR} if the inequality is strict whenever the consequent sets are disjoint. Thus GE and GR mean nonstrict and strict expansion, respectively. A stochastic GR matrix is called a \emph{Sarymsakov matrix}, or simply \emph{Sarymsakov}; the class is denoted by $\Tzero$.

Let $e_i$ be the $i$th canonical column vector and $\|\cdot\|_1$ the vector $\ell_1$ norm. For $P\in\Sn$ and a column vector $x\in\mathbb R^n$, define the \emph{Dobrushin coefficient} and disagreement seminorm by
\begin{align}
 \delta(P)&:=\frac12\max_{i,j}\|e_i^{\mathsf T}P-e_j^{\mathsf T}P\|_1,
 \label{eq:delta}\\
 \Delta(x)&:=\max_i x_i-\min_i x_i.
 \label{eq:disagreement}
\end{align}
The following standard properties are included for completeness; see also \cite{seneta1981,hartfiel2002}.

\begin{lemma}[Dobrushin contraction]
\label{lem:dobrushin}
For $P,Q\in\Sn$ and $x\in\mathbb R^n$,
\begin{equation}
 \Delta(Px)\le\delta(P)\Delta(x),
 \qquad
 \delta(PQ)\le\delta(P)\delta(Q).
 \label{eq:dobrushinproperties}
\end{equation}
Moreover, a stochastic matrix is scrambling if and only if its Dobrushin coefficient is strictly less than one.
\end{lemma}
% skip the proof
%\begin{IEEEproof}
%Let $p_i^{\mathsf T}=e_i^{\mathsf T}P$, $m=\min_j x_j$, and $M=\max_j x_j$. Since $(p_i-p_j)^{\mathsf T}\one=0$, the positive and %negative parts of $p_i-p_j$ have the same total mass. Hence
%\[
% |(Px)_i-(Px)_j|
 %\le \tfrac12\|p_i-p_j\|_1(M-m).
%\]
%Maximization over $i,j$ gives the first inequality in \eqref{eq:dobrushinproperties}.

%For probability rows $u,v$, let $d=\|u-v\|_1/2$. If $d=0$, the claim is immediate. Otherwise, set $w_i=\min\{u_i,v_i\}$ and normalize the disjoint residuals as $\widehat u=(u-w)/d$ and $\widehat v=(v-w)/d$. Since $\widehat uQ$ and $\widehat vQ$ are convex combinations of the rows of $Q$,
%\[
 %\|uQ-vQ\|_1
 %=d\|\widehat uQ-\widehat vQ\|_1
 %\le 2d\,\delta(Q)
 %=\delta(Q)\|u-v\|_1.
%\]
%Applying this estimate to $u=e_i^{\mathsf T}P$ and $v=e_j^{\mathsf T}P$, and then maximizing, proves the second inequality. Finally,
%\[
% \delta(P)=1-\min_{i,j}\sum_{\ell=1}^n\min\{p_{i\ell},p_{j\ell}\},
%\]
%so $\delta(P)<1$ exactly when every row pair has a common positive column.
%\end{IEEEproof}

It follows from Lemma~\ref{lem:dobrushin} that scrambling is preserved by stochastic multiplication on either side. Let $\mathcal C^h$ be the set of all $h$-factor products from $\mathcal C\subseteq\Sn$, and let $\mathcal S_{\rm scr}$ be the scrambling class within $\Sn$. Its \emph{uniform scrambling horizon}, when finite, is
\begin{equation}
 \hscr(\mathcal C):=\min\{h\ge1:\mathcal C^h\subseteq\mathcal S_{\rm scr}\}.
 \label{eq:hscr}
\end{equation}
Thus a stated horizon is exact only when minimality is proved. We also use the standard facts that $\Tzero$ is a semigroup, every product of $n-1$ Sarymsakov matrices is scrambling, and a compact family is a consensus set if every finite product is SIA \cite{seneta1981,hartfielseneta1990,hartfiel2002,wolfowitz1963}.

\begin{lemma}[Support monotonicity and composition]
\label{lem:support}
The following statements hold.
\begin{enumerate}
\item Let $A$ and $B$ be row-allowable square nonnegative matrices with $A\succeq B$. If $B$ is GE, GR, or scrambling, then $A$ has the same property.
\item Let $P,Q\in\Sn$ with $P\succeq Q$. If $Q$ is SIA, then $P$ is SIA.
\item Let $A,B$ be row-allowable square nonnegative matrices and $M\succeq AB$. If $A$ and $B$ are GE, then $M$ is GE. If one factor is GR and the other is GE, then $M$ is GR. The same conclusions hold for $M\succeq BA$.
\end{enumerate}
\end{lemma}
\begin{IEEEproof}
For part 1, consequent-set inclusion gives $F_A(I)\supseteq F_B(I)$ for every $I$. An intersection present for $B$ remains present for $A$; when the consequent sets for $A$ are disjoint, the corresponding sets for $B$ are also disjoint, and the required cardinality inequality transfers directly. The scrambling claim is the special case of singleton row sets.

For part~2, note that $Q$ is SIA if there exists an integer
$k\geq1$ such that $Q^k$ is scrambling~\cite[Theorem~5]{hsu2023}. If $P\succeq Q$ then $P^k\succeq Q^k$, proving $P$ is SIA as well since scrambling is preserved under support enlargement.

For part 3, suppose $F_M(I)$ and $F_M(J)$ are disjoint. Then $F_{AB}(I)$ and $F_{AB}(J)$ are disjoint. Row-allowability of $B$ implies that $F_A(I)$ and $F_A(J)$ are disjoint: otherwise, a common index would have a positive consequent under $B$, producing an intersection after composition. Applying the expansion conditions first to $A$ and then to $B$ gives
\begin{align*}
 |F_M(I)\cup F_M(J)|
 &\ge |F_{AB}(I)\cup F_{AB}(J)|\\
 &\ge |F_A(I)\cup F_A(J)|\\
 &\ge |I\cup J|.
\end{align*}
At least one inequality is strict when one factor is GR. The proof for $BA$ is identical.
\end{IEEEproof}

\begin{remark}[Why nonnegative blocks are allowed]
\label{rem:supportonly}
The GE, GR, and scrambling properties depend only on support. This is important below because a leading principal block of a stochastic matrix is generally substochastic. Row normalization preserves its support and therefore preserves all three properties. By contrast, the SIA statement in Lemma~\ref{lem:support} is restricted to stochastic matrices because it concerns powers and terminal probability flow, not only the zero--nonzero pattern.
\end{remark}

\begin{lemma}[Products of nonnegative GR blocks]
\label{lem:normalized}
Let $A_1,\ldots,A_{r-1}$ be row-allowable nonnegative $r\times r$ matrices with the GR property, where $r\ge2$. Then $A_1\cdots A_{r-1}$ is scrambling.
\end{lemma}
\begin{IEEEproof}
Set $\widehat A_i=\operatorname{diag}(A_i\one)^{-1}A_i$. Each $\widehat A_i$ is stochastic, has the same support as $A_i$, and is Sarymsakov. Hence $\widehat A_1\cdots\widehat A_{r-1}$ is scrambling. For nonnegative factors, a product entry is positive exactly when there is a compatible positive path through the factor supports. Factorwise identical supports therefore give identical product supports, and $A_1\cdots A_{r-1}$ is scrambling as well.
\end{IEEEproof}

\begin{lemma}[Compact-family estimate]
\label{lem:compact}
Suppose every product of $h$ matrices from a class $\mathcal C\subseteq\Sn$ is scrambling. For a compact $\mathcal P\subseteq\mathcal C$, define
\begin{equation}
 \theta_{\mathcal P,h}:=\max_{P_1,\ldots,P_h\in\mathcal P}
 \delta(P_h\cdots P_1).
 \label{eq:theta}
\end{equation}
Then $\theta_{\mathcal P,h}<1$, and every sequence in $\mathcal P$ satisfies
\begin{equation}
 \delta(P(k-1)\cdots P(0))
 \le \theta_{\mathcal P,h}^{\lfloor k/h\rfloor}.
 \label{eq:compactrate}
\end{equation}
Consequently, $\Delta(x(k))\le\theta_{\mathcal P,h}^{\lfloor k/h\rfloor}\Delta(x(0))$.
\end{lemma}
\begin{IEEEproof}
The product map and $\delta$ are continuous, so the maximum in \eqref{eq:theta} is attained on the compact set $\mathcal P^h$. Every maximizing product is scrambling, whence the maximum is strictly below one. Partition a product of length $k$ into $\lfloor k/h\rfloor$ consecutive full blocks and one incomplete block. The latter has Dobrushin coefficient at most one, and repeated use of submultiplicativity gives \eqref{eq:compactrate}. The disagreement estimate follows from \eqref{eq:disagreement}.
\end{IEEEproof}

\section{A Power-Generated Support Upper Hull}
The relation $A\succeq B$ is a preorder on weighted matrices and becomes a partial order after matrices with the same support are identified. For $A\in\Sn$, define the \emph{principal support upper set}
\begin{equation}
 \uparrow A:=\{P\in\Sn:\supp(A)\subseteq\supp(P)\}.
 \label{eq:principalupper}
\end{equation}
Thus $\uparrow A$ contains all stochastic matrices that retain every positive position of $A$, while allowing additional positive positions and different positive weights. For a collection $\mathcal G\subseteq\Sn$, define its \emph{support upper hull} by
\begin{equation}
 \uparrow\mathcal G:=\bigcup_{A\in\mathcal G}\uparrow A.
 \label{eq:upperhull}
\end{equation}
These are support-inclusion notions; they are not convex cones under matrix addition or scalar multiplication.

Let $R\in\Sn$ be SIA and GE, and define
\begin{equation}
 K(R):=\{k\ge1:R^k\notin\Tzero\}.
 \label{eq:KR}
\end{equation}
Every power of $R$ is GE by Lemma~\ref{lem:support}. Because a sufficiently large power of an SIA matrix is scrambling, and therefore Sarymsakov, $K(R)$ is finite. Moreover, if $R^k$ is GR, then $R^{k+\ell}=R^kR^\ell$ is GR for every $\ell\ge1$, because $R^\ell$ is GE. Thus $K(R)$ is empty or an initial segment $\{1,\ldots,h_R\}$.

The class
\begin{equation}
 \E(R):=\Tzero\cup\bigcup_{k\in K(R)}
 \{P\in\Sn:\supp(P)=\supp(R^k)\}
 \label{eq:E}
\end{equation}
is the equal-pattern semigroup introduced in \cite{hsu2023}, after omitting power patterns whose equal-support classes already lie in $\Tzero$. It is therefore the existing baseline. The new class proposed here is
\begin{equation}
 \U(R):=\Tzero\cup\uparrow\{R^k:k\ge1\}
 =\Tzero\cup\bigcup_{k\ge1}\uparrow R^k.
 \label{eq:U}
\end{equation}
Thus $\U(R)$ is the support upper hull of all positive powers of $R$: it replaces exact pattern equality by support inclusion and may strictly contain $\E(R)$.

\begin{theorem}[Finite representation and closure]
\label{thm:upper}
The class $\U(R)$ admits the finite representation
\begin{equation}
 \U(R)=\Tzero\cup\bigcup_{k\in K(R)}\uparrow R^k.
 \label{eq:Ufinite}
\end{equation}
It is a semigroup, every compact subset is a consensus set, and $\E(R)\subseteq\U(R)$.
\end{theorem}
\begin{IEEEproof}
If $k\notin K(R)$, then $R^k$ is GR. Lemma~\ref{lem:support} gives $\uparrow R^k\subseteq\Tzero$, so powers outside the finite set $K(R)$ add no new matrices and \eqref{eq:Ufinite} follows.

For closure, let $P,Q\in\U(R)$. If both belong to $\Tzero$, then $PQ\in\Tzero$. If $P$ is GR and $Q\succeq R^b$, then $Q$ is GE and Lemma~\ref{lem:support}, applied directly to the product, shows that both $PQ$ and $QP$ are GR. The same argument covers the symmetric case. Finally, if $P\succeq R^a$ and $Q\succeq R^b$, then nonnegativity gives
\begin{equation}
 PQ\succeq R^{a+b}.
 \label{eq:powerproduct}
\end{equation}
If $a+b\in K(R)$, the product belongs to the corresponding principal support upper set; otherwise $R^{a+b}$ is GR and support monotonicity places $PQ$ in $\Tzero$. Hence $\U(R)$ is closed. Each member is SIA because it is Sarymsakov or support-dominates an SIA power of $R$. Closure therefore implies that every finite product from a compact subset is SIA, and the standard compact finite-product criterion gives consensus. The inclusion $\E(R)\subseteq\U(R)$ is immediate from equality implying support inclusion.
\end{IEEEproof}

\begin{example}[Strict relaxation of equal supports]
\label{ex:upper}
Let
\begin{equation}
 R=\begin{bmatrix}1&0&0\\[1mm]1/2&1/2&0\\[1mm]0&1&0\end{bmatrix},
 \qquad
 P=\begin{bmatrix}1&0&0\\[1mm]1/3&1/3&1/3\\[1mm]0&1&0\end{bmatrix}.
 \label{eq:RP}
\end{equation}
The matrix $R$ is SIA and GE but not GR, whereas $R^2$ is scrambling; hence $K(R)=\{1\}$. Moreover, $P\succeq R$ with unequal supports. For $I=\{1\}$ and $J=\{3\}$, $F_P(I)=\{1\}$ and $F_P(J)=\{2\}$, so $P\notin\Tzero$. Therefore $P\in\U(R)\setminus\E(R)$.
\end{example}

\begin{remark}[Interpretation of the support enlargement]
The class $\E(R)$ requires exact equality with one of the relevant power-support patterns, whereas $\U(R)$ retains every positive position of a power $R^k$ and permits additional positions and redistributed row weights. The maximality statement in Theorem~\ref{thm:upper} is relative to these principal support upper sets generated by the powers of the same prescribed matrix $R$; it is not a claim that $\U(R)$ is the largest possible SIA semigroup containing $\E(R)$. Counterexamples will be given in Remark~\ref{rem:cxpm} to show that the claim is false.

Support inclusion should also not be confused with an additive perturbation. If $P$ and $R^k$ are stochastic and $P=R^k+A$ with $A\ge0$, then $A\one=P\one-R^k\one=0$, and nonnegativity forces $A=0$. A nontrivial enlargement is therefore obtained by changing the support pattern and redistributing row weights, not by adding a nonnegative matrix entrywise.
\end{remark}

\section{A Fixed-Core Block Semigroup}
Fix $r\in\{1,\ldots,n\}$ and partition
\begin{equation}
 P=\begin{bmatrix}P_{11}&P_{12}\\ P_{21}&P_{22}\end{bmatrix}\in\Sn,
 \label{eq:block}
\end{equation}
where $P_{11}\in\mathbb R^{r\times r}$ is the leading principal block and $P_{21}\in\mathbb R^{(n-r)\times r}$ is the lower-left block. Note that if \(P_{12}\) contains positive entries, \(P_{11}\) becomes substochastic.
\begin{definition}
The class $\TB(r)$, for $r\in\{1,2,\cdots,n-1\}$, consists of the matrices $P\in\Sn$ for which $P_{11}$ has the GR property and $P_{21}$ is row-allowable; $\TB(n)=\Tzero$.
\end{definition}

\begin{lemma}
\label{lem:TBsia}
Every matrix in $\TB(r)$ is SIA.
\end{lemma}
\begin{IEEEproof}
Let $\mathcal C$ be a terminal class of the full graph. It must intersect the first $r$ states: if $\mathcal C$ contained only lower states, row-allowability of $P_{21}$ would give an edge from every row of $\mathcal C$ into the core, contradicting closedness. Suppose two terminal classes existed, and let $I$ and $J$ be their respective intersections with the core. These sets are nonempty and disjoint. Closedness of the terminal classes implies
\[
 F_{P_{11}}(I)\subseteq I,
 \qquad
 F_{P_{11}}(J)\subseteq J.
\]
Thus the two consequent sets are disjoint and their union has cardinality at most $|I\cup J|$, contradicting the GR property. The full graph therefore has a unique terminal class.

Let $\widehat P_{11}=\operatorname{diag}(P_{11}\one)^{-1}P_{11}$. This row-normalized matrix is Sarymsakov, so its graph has a unique aperiodic terminal component $\mathcal D$. The intersection of the full terminal class with the core is a nonempty $P_{11}$-closed set and hence contains $\mathcal D$. All cycles in $\mathcal D$ remain cycles of the full terminal class. Their lengths have greatest common divisor one, so the full terminal class is aperiodic. The stochastic matrix $P$ is therefore SIA.
\end{IEEEproof}

\begin{theorem}
\label{thm:TBclosure}
For every fixed $r$, the class $\TB(r)$ is a semigroup, and every compact subset is a consensus set. It is also support-upward: if $P\in\TB(r)$ and $Q\in\Sn$ satisfies $Q\succeq P$, then $Q\in\TB(r)$.
\end{theorem}
\begin{IEEEproof}
For $P,Q\in\TB(r)$, the leading block of the product is
\begin{equation}
 (PQ)_{11}=P_{11}Q_{11}+P_{12}Q_{21}\succeq P_{11}Q_{11}.
 \label{eq:leadproduct}
\end{equation}
Both $P_{11}$ and $Q_{11}$ are GR and therefore GE. Lemma~\ref{lem:support} makes the product $P_{11}Q_{11}$ GR, and support monotonicity transfers this property to $(PQ)_{11}$. For the lower-left block,
\begin{equation}
 (PQ)_{21}=P_{21}Q_{11}+P_{22}Q_{21}.
 \label{eq:lowerproduct}
\end{equation}
Each row of $P_{21}$ has a positive entry in some core column, and the corresponding row of the row-allowable matrix $Q_{11}$ has a positive consequent. Hence $P_{21}Q_{11}$, and therefore $(PQ)_{21}$, is row-allowable. This proves $PQ\in\TB(r)$. Lemma~\ref{lem:TBsia} and the compact finite-product characterization give the consensus claim. Finally, support enlargement preserves both the GR property of $P_{11}$ and row-allowability of $P_{21}$.
\end{IEEEproof}

\begin{theorem}[Exact fixed-core scrambling horizon]
\label{thm:TBhorizon}
Assume $n\ge2$ and set
\begin{equation}
 h_{n,r}:=\min\{r,n-1\}.
 \label{eq:hnr}
\end{equation}
Then
\begin{equation}
 \hscr(\TB(r))=h_{n,r}.
 \label{eq:exactTB}
\end{equation}
Consequently, every compact $\mathcal P\subseteq\TB(r)$ satisfies \eqref{eq:compactrate} with $h=h_{n,r}$.
\end{theorem}
\begin{IEEEproof}
We first prove the upper bound. Suppose $r<n$, and let $P^{(1)},\ldots,P^{(r)}\in\TB(r)$. Partition every factor according to the same first $r$ states:
\[
 P^{(j)}=
 \begin{bmatrix}
 P_{11}^{(j)}&P_{12}^{(j)}\\
 P_{21}^{(j)}&P_{22}^{(j)}
 \end{bmatrix},
 \qquad j=1,\ldots,r.
\]
For an initial row $i\le r$, the GR block $P_{11}^{(1)}$ is row-allowable; for $i>r$, the block $P_{21}^{(1)}$ is row-allowable by the definition of $\TB(r)$. Hence there exists a core index $a_i\in\{1,\ldots,r\}$ such that $(P^{(1)})_{ia_i}>0$. If $r=1$, all rows already share the first column. If $r\ge2$, define the core product
\[
 B:=P_{11}^{(2)}P_{11}^{(3)}\cdots P_{11}^{(r)}.
\]
It is scrambling by Lemma~\ref{lem:normalized}. Given initial rows $i$ and $j$, rows $a_i$ and $a_j$ of $B$ share a positive column. Combining the first-step entries $(P^{(1)})_{ia_i}$ and $(P^{(1)})_{ja_j}$ with the corresponding positive core paths shows that rows $i$ and $j$ of $P^{(1)}\cdots P^{(r)}$ share that column. Thus $\hscr(\TB(r))\le r$ when $r<n$. If $r=n$, then $\TB(n)=\Tzero$, and the classical Sarymsakov result gives $\hscr(\TB(n))\le n-1$.

It remains to prove minimality. For $q\ge2$, let $C_q$ be the stochastic matrix defined by
\begin{equation}
 e_1^{\mathsf T}C_q=e_1^{\mathsf T},\qquad
 e_i^{\mathsf T}C_q=\tfrac12(e_{i-1}^{\mathsf T}+e_i^{\mathsf T}),
 \quad 2\le i\le q.
 \label{eq:Cq}
\end{equation}
For a nonempty set $I\subseteq\{1,\ldots,q\}$,
\[
 F_{C_q}(I)=I\cup\{i-1:i\in I,\ i\ge2\}.
\]
Equality $|F_{C_q}(I)|=|I|$ is possible only when $I$ is an initial segment $\{1,\ldots,t\}$. Two disjoint nonempty sets cannot both be initial segments. Therefore, whenever $F_{C_q}(I)$ and $F_{C_q}(J)$ are disjoint, at least one of them expands strictly, and $C_q$ is GR.

For $2\le r<n$, define
\begin{equation}
 P_{n,r}:=
 \begin{bmatrix}
 C_r&0_{r\times(n-r)}\\
 \one_{n-r}e_r^{\mathsf T}&0_{(n-r)\times(n-r)}
 \end{bmatrix}\in\TB(r).
 \label{eq:sharpTB}
\end{equation}
The first row of $P_{n,r}^{r-1}$ is supported on $\{1\}$, whereas every lower row is supported on
$\supp(e_r^{\mathsf T}C_r^{r-2})=\{2,\ldots,r\}$. Thus $P_{n,r}^{r-1}$ is not scrambling, proving $\hscr(\TB(r))\ge r$. For $r=n\ge3$, the first and last rows of $C_n^{n-2}$ have supports $\{1\}$ and $\{2,\ldots,n\}$, respectively, so the classical bound $n-1$ is also minimal. The cases in which $h_{n,r}=1$ are immediate from the definition of a positive horizon. Lemma~\ref{lem:compact} gives the compact-family estimate.
\end{IEEEproof}

\begin{remark}[Network interpretation]
For $r<n$, one factor sends every row into the ordered core and the next $r-1$ core blocks mix the possible entries. The exact horizon is support determined even when weights and supports vary and probability leaves the core; compactness is needed only for the numerical factor $\theta_{\mathcal P,h_{n,r}}$. Thus a small core gives a contraction block much shorter than the ambient-order bound. In consensus terminology, the first $r$ agents form an ordered mixing core, while every remaining agent assigns positive weight to at least one core agent at each time. Links outside the core may switch arbitrarily, and links may also leave the core, provided the defining support conditions remain valid.
\end{remark}

\begin{example}[A compact switched family outside the Sarymsakov class]
\label{ex:switchedfamily}
For $0\le\gamma\le1/2$, define
\begin{equation}
 P(\gamma):=
 \begin{bmatrix}
 1/2&0&0&1/2\\
 1/2&0&\gamma&1/2-\gamma\\
 0&1&0&0\\
 1/2&0&0&1/2
 \end{bmatrix}.
 \label{eq:switchedfamily}
\end{equation}
Let $\mathcal P_{\rm sw}:=\{P(\gamma):0\le\gamma\le1/2\}$. The leading $2\times2$ block of every $P(\gamma)$ is GR, and its lower-left block
\[
 \begin{bmatrix}0&1\\[1mm]1/2&0\end{bmatrix}
\]
is row-allowable. Hence $\mathcal P_{\rm sw}\subset\TB(2)$. On the other hand, for $I=\{3\}$ and $J=\{1,4\}$,
\[
 F_{P(\gamma)}(I)=\{2\},
 \qquad
 F_{P(\gamma)}(J)=\{1,4\},
\]
so the consequent sets are disjoint and their union has cardinality $3=|I\cup J|$. Thus every member of $\mathcal P_{\rm sw}$ lies outside $\Tzero$.

For $a,b\in[0,1/2]$,
\begin{equation}
 P(a)P(b)=
 \begin{bmatrix}
 1/2&0&0&1/2\\
 (1-a)/2&a&0&(1-a)/2\\
 1/2&0&b&1/2-b\\
 1/2&0&0&1/2
 \end{bmatrix}.
 \label{eq:switchedproduct}
\end{equation}
A direct row comparison gives
\begin{equation}
 \delta(P(a)P(b))
 =\max\{a,b,a/2+b\}
 \le\frac34,
 \label{eq:switchedrate}
\end{equation}
with equality at $a=b=1/2$. Therefore, for arbitrary switching $P(k)=P(\gamma_k)$ with $\gamma_k\in[0,1/2]$,
\begin{equation}
 \Delta(x(k))
 \le\left(\frac34\right)^{\lfloor k/2\rfloor}\Delta(x(0)).
 \label{eq:switchedconsensus}
\end{equation}
This family makes the control implication concrete: none of the individual switching matrices is Sarymsakov, yet the fixed-core theorem gives uniform consensus, and the exact two-step coefficient is computable directly.
\end{example}

\begin{example}[Why row-allowability is needed]
\label{ex:rowallowable}
Let
\begin{equation}
 P_0=\begin{bmatrix}1/2&0&1/2\\0&1&0\\0&1&0\end{bmatrix},
 \quad
 Q_0=\begin{bmatrix}1&0&0\\0&1/2&1/2\\1&0&0\end{bmatrix}.
 \label{eq:rowcounter}
\end{equation}
Both matrices are SIA and their leading block of order one is GR. Their lower-left blocks at order one are not row-allowable, however, and
\[
 P_0Q_0=\begin{bmatrix}1&0&0\\0&1/2&1/2\\0&1/2&1/2\end{bmatrix}
\]
has the two terminal classes $\{1\}$ and $\{2,3\}$. The failure is rowwise: it is not enough that some lower state reaches the core; every lower row must do so to rule out a closed lower group.
\end{example}

\section{Nested Expansion Compatibility and Almost Sarymsakov Matrices}
The union $\bigcup_r\TB(r)$ permits different core sizes but is not automatically closed: a factor with a smaller core need not satisfy any expansion property at the larger order selected by another factor. It is useful to separate two roles. Nested GR/GE conditions provide algebraic compatibility under multiplication, while row-allowable access to the selected core supplies the SIA and consensus properties. The nested-expansion class introduced below is therefore an auxiliary algebraic device: its purpose is to isolate the closure mechanism, not to define an additional consensus class.

For $P\in\Sn$ and $s\in\{1,\ldots,n\}$, write
\begin{equation}
 P=
 \begin{bmatrix}
 P^{[s]}&P_{12}^{[s]}\\
 P_{21}^{[s]}&P_{22}^{[s]}
 \end{bmatrix},
 \label{eq:sblock}
\end{equation}
where $P^{[s]}$ is the leading principal block of order $s$. When $s=n$, $P_{21}^{[s]}$ has zero rows and its row-allowability condition is vacuous.

\begin{definition}[Nested-expansion classes]
\label{def:TNE}
For $r\in\{1,\ldots,n\}$, let $\TNE(r)$ consist of the matrices $P\in\Sn$ for which $P^{[r]}$ is GR and $P^{[s]}$ is GE for every $s>r$. Define
\begin{equation}
 \TNE:=\bigcup_{r=1}^n\TNE(r),
 \qquad
 \ellne(P):=\min\{r:P\in\TNE(r)\}.
 \label{eq:TNE}
\end{equation}
A member of $\TNE(r)$ is called expansion-compatible at level $r$.
\end{definition}

\begin{proposition}[Nested-expansion closure]
\label{prop:TNEcompatibility}
If $P\in\TNE(r_P)$ and $Q\in\TNE(r_Q)$, then
\begin{equation}
 PQ,QP\in\TNE(\max\{r_P,r_Q\}).
 \label{eq:TNEcompatibility}
\end{equation}
Consequently, $\TNE$ is a semigroup and
\begin{equation}
 \ellne(PQ)\le\max\{\ellne(P),\ellne(Q)\}.
 \label{eq:TNElevel}
\end{equation}
\end{proposition}
\begin{IEEEproof}
Let $r=\max\{r_P,r_Q\}$. At order $r$, one of $P^{[r]}$ and $Q^{[r]}$ is GR and the other is GE; if $r_P=r_Q$, both are GR. From the block multiplication associated with \eqref{eq:sblock},
\[
 (PQ)^{[r]}
 =P^{[r]}Q^{[r]}+P_{12}^{[r]}Q_{21}^{[r]}
 \succeq P^{[r]}Q^{[r]}.
\]
Lemma~\ref{lem:support} therefore makes $(PQ)^{[r]}$ GR. For every $s>r$, both $P^{[s]}$ and $Q^{[s]}$ are GE, and
\[
 (PQ)^{[s]}\succeq P^{[s]}Q^{[s]}.
\]
The same lemma makes $(PQ)^{[s]}$ GE. Hence $PQ\in\TNE(r)$. Interchanging the factors proves the assertion for $QP$. Closure of the union and \eqref{eq:TNElevel} follow immediately.
\end{IEEEproof}

\begin{remark}[Why access is a separate requirement]
The nested GR/GE conditions are sufficient for multiplication closure but not for ergodicity. For every $n\ge2$, the identity matrix satisfies $I_n\in\TNE(1)$: its leading block of order one is GR vacuously, and every larger leading identity block is GE. Nevertheless, $I_n$ is not SIA, scrambling, or consensus generating. Thus row-allowable lower-left access is unnecessary for the bare semigroup property of $\TNE$, but it is essential for selecting a class whose members reach the mixing core and inherit SIA and finite-scrambling guarantees.
\end{remark}

\begin{definition}[Almost Sarymsakov classes]
\label{def:TAS}
For $r\in\{1,\ldots,n\}$, define the accessible subclass
\begin{equation}
 \TAS(r):=\{P\in\TNE(r):P_{21}^{[r]}\text{ is row-allowable}\}.
 \label{eq:TASr}
\end{equation}
Set
\begin{equation}
 \TAS:=\bigcup_{r=1}^n\TAS(r),
 \qquad
 \ellv(P):=\min\{r:P\in\TAS(r)\}.
 \label{eq:TAS}
\end{equation}
A member of $\TAS$ is called an \emph{almost Sarymsakov matrix}, or simply \emph{almost Sarymsakov}.
\end{definition}

The integer $\ellv(P)$ is the smallest leading order at which strict expansion and row-allowable access hold while all larger leading blocks remain GE. Equivalently, $\TAS(r)=\TNE(r)\cap\TB(r)$. The conditions depend only on support: at order $s$, direct enumeration uses at most $3^s$ assignments to $I$, $J$, or neither, so scanning $r$ computes $\ellv(P)$ from the switching topology. Since $\TAS(n)=\Tzero$, the classical class is contained in $\TAS$, and each accessible level is support-upward.

\begin{proposition}[Accessible variable-core compatibility]
\label{prop:TAScompatibility}
If $P\in\TAS(r_P)$ and $Q\in\TAS(r_Q)$, then
\begin{equation}
 PQ,QP\in\TAS(\max\{r_P,r_Q\}).
 \label{eq:TAScompatibility}
\end{equation}
Consequently,
\begin{equation}
 \ellv(PQ)\le\max\{\ellv(P),\ellv(Q)\}.
 \label{eq:level}
\end{equation}
\end{proposition}
\begin{IEEEproof}
Let $r=\max\{r_P,r_Q\}$. Proposition~\ref{prop:TNEcompatibility} gives $PQ,QP\in\TNE(r)$, so only access at order $r$ remains to be verified. Because $r_P\le r$, every row below $r$ is also below $r_P$ and therefore has a positive entry in the first $r_P$ columns of $P$; hence $P_{21}^{[r]}$ is row-allowable. The leading block $Q^{[r]}$ is GE or GR and thus row-allowable. The lower-left block of $PQ$ satisfies
\begin{equation}
 (PQ)_{21}^{[r]}
 =P_{21}^{[r]}Q^{[r]}+P_{22}^{[r]}Q_{21}^{[r]}.
 \label{eq:TASaccessproduct}
\end{equation}
The first term is row-allowable, so $(PQ)_{21}^{[r]}$ is row-allowable. Thus $PQ\in\TAS(r)$. Interchanging the factors proves the assertion for $QP$, and the level inequality follows by minimality.
\end{IEEEproof}

\begin{theorem}
\label{thm:TASsemigroup}
The class $\TAS$ is a semigroup, and every compact subset is a consensus set. Moreover,
\begin{equation}
 \TAS=\Tzero\quad(n\le2),
 \qquad
 \Tzero\subsetneq\TAS\quad(n\ge3).
 \label{eq:strictdimension}
\end{equation}
\end{theorem}
\begin{IEEEproof}
Closure follows from Proposition~\ref{prop:TAScompatibility}. Every member belongs to some $\TB(r)$ and is SIA by Lemma~\ref{lem:TBsia}; closure then makes every finite product SIA, so the compact-consensus criterion applies. For $n=1$ the claim is immediate. For $n=2$, every matrix in $\TAS(1)\subseteq\TB(1)$ is scrambling because all rows share the first column, and $\TAS(2)=\Tzero$; hence no enlargement occurs.

For $n\ge3$, define $A_n$ by
\begin{equation}
 e_i^{\mathsf T}A_n=e_1^{\mathsf T}\ (i=1,2),
 \qquad
 e_i^{\mathsf T}A_n=\tfrac12e_2^{\mathsf T}+\tfrac12e_i^{\mathsf T}\ (i\ge3).
 \label{eq:An}
\end{equation}
Its leading block of order two is GR because its two row supports coincide, and each lower row has a positive entry in column two. Consider a leading block of order $s\ge3$. If two disjoint row sets have disjoint consequent sets, one set must lie in $\{1,2\}$ and the other in $\{3,\ldots,s\}$; otherwise the consequent sets intersect at column one or column two. If the latter set is $J$, its consequent set is $\{2\}\cup J$, while the former consequent set is $\{1\}$. Their union has $|J|+2$ elements, at least the number of selected rows because the first set has at most two elements. Thus every larger leading block is GE. Taking the first set to be $\{1,2\}$ and $J=\{3\}$ gives equality, not strict expansion, so $A_n\notin\Tzero$. Therefore $A_n\in\TAS(2)\setminus\Tzero$.
\end{IEEEproof}

\begin{theorem}[Universal scrambling-horizon bound]
\label{thm:universal}
Assume $n\ge2$. For $m\in\{1,\ldots,n-1\}$, set
\[
 \TAS^{[m]}:=\bigcup_{r=1}^m\TAS(r).
\]
Every product of at least $m!$ matrices from $\TAS^{[m]}$ is scrambling. Furthermore, every product of at least
\begin{equation}
 L_n:=(n-1)(n-1)!
 \label{eq:Ln}
\end{equation}
almost Sarymsakov matrices of order $n$ is scrambling. Equivalently,
\begin{equation}
 \hscr(\TAS)\le L_n.
 \label{eq:TAShorizon}
\end{equation}
\end{theorem}
\begin{IEEEproof}
It is enough to prove the two stated threshold lengths because stochastic multiplication preserves scrambling. We first prove the truncated assertion by induction on $m$. For $m=1$, every factor in $\TAS(1)$ is scrambling by Theorem~\ref{thm:TBhorizon}. Assume the assertion holds for $m-1$, and consider a product of $m!$ factors from $\TAS^{[m]}$. Let $q$ be the number of factors whose minimum level is exactly $m$.

If $q\ge m$, choose $m$ of those factors in their product order and cut the product into $m$ consecutive nonempty subproducts, each containing one chosen level-$m$ factor. All other factors in a subproduct have level at most $m$. Repeated application of Proposition~\ref{prop:TAScompatibility} places each subproduct in $\TAS(m)\subseteq\TB(m)$. The product of these $m$ block factors is scrambling by Theorem~\ref{thm:TBhorizon}.

If $q\le m-1$, delete the level-$m$ factors. The remaining factors form at most $q+1\le m$ consecutive runs, all drawn from $\TAS^{[m-1]}$. Were every run shorter than $(m-1)!$, the total number of factors would be at most
\begin{equation}
 q+(q+1)((m-1)!-1)\le m!-1,
 \label{eq:runcount}
\end{equation}
a contradiction. Hence one run contains a consecutive subproduct of $(m-1)!$ lower-level factors. That subproduct is scrambling by the induction hypothesis, and multiplication by the factors before and after it preserves scrambling. This proves the truncated assertion.

Now consider $L_n$ arbitrary factors from $\TAS$, and let $q$ be the number of level-$n$ factors. Since $\TAS(n)=\Tzero$, if $q\ge n-1$, the product can be cut into $n-1$ consecutive subproducts, each containing a level-$n$ factor. Proposition~\ref{prop:TAScompatibility} places every subproduct in $\Tzero$, and the classical $n-1$ factor result makes their product scrambling. If $q\le n-2$, deletion leaves at most $q+1\le n-1$ runs from $\TAS^{[n-1]}$. If every run had fewer than $(n-1)!$ factors, then
\begin{equation}
 L_n\le(n-2)+(n-1)((n-1)!-1)=L_n-1,
\end{equation}
which is impossible. A run of at least $(n-1)!$ factors is scrambling by the truncated assertion with $m=n-1$, and therefore so is the full product.
\end{IEEEproof}

By Lemma~\ref{lem:compact}, every compact $\mathcal P\subseteq\TAS$ has $\theta_{\mathcal P,L_n}<1$ and satisfies \eqref{eq:compactrate} with $h=L_n$.

\begin{remark}[Origin and sharpness of the bound]
The factorial reflects a worst-case alternation of levels: either enough highest-level factors form fixed-core blocks, or a long lower-level run occurs. The bound is not claimed sharp in general, but it is sharp for $n=3$, where $L_3=4$. Indeed, consider
\[
 A_1=\begin{bmatrix}
 1&0&0\\
 1/2&1/2&0\\
 0&1&0
 \end{bmatrix},
 \qquad
 A_2=\begin{bmatrix}
 1/2&1/2&0\\
 0&0&1\\
 1/2&0&1/2
 \end{bmatrix},
\]
\[
 A_3=\begin{bmatrix}
 1&0&0\\
 1&0&0\\
 0&1/2&1/2
 \end{bmatrix}.
\]
These matrices satisfy $A_1,A_3\in\TAS(2)$ and $A_2\in\TAS(3)$, but
\[
 A_1A_2A_3=\begin{bmatrix}
 1&0&0\\
 1/2&1/4&1/4\\
 0&1/2&1/2
 \end{bmatrix}
\]
is not scrambling. Hence three factors do not suffice, whereas Theorem~\ref{thm:universal} guarantees scrambling after four factors.
\end{remark}

\begin{example}[Why higher-order GE is needed]
\label{ex:higherGE}
Let
\[
 P_b=\begin{bmatrix}
 1&0&0\\
 1&0&0\\
 0&1&0
 \end{bmatrix}\in\TB(2),
\]
\[
 Q_b=\begin{bmatrix}
 1/2&1/2&0\\
 0&0&1\\
 1/2&0&1/2
 \end{bmatrix}\in\TB(3).
\]
The full block $P_b$ is not GE because the disjoint sets $\{1,2\}$ and $\{3\}$ have the consequent sets $\{1\}$ and $\{2\}$, respectively. Moreover,
\[
 P_bQ_b=\begin{bmatrix}
 1/2&1/2&0\\
 1/2&1/2&0\\
 0&0&1
 \end{bmatrix}
\]
has the two terminal classes $\{1,2\}$ and $\{3\}$. Therefore $\bigcup_r\TB(r)$ need not be a semigroup. The matrix $P_b$ fails the larger-order GE requirement and therefore does not belong to $\TNE(2)$. The higher-order GE conditions in Definition~\ref{def:TNE} supply the algebraic compatibility needed at the larger core order; row-allowable access is a separate requirement imposed later in Definition~\ref{def:TAS}.
\end{example}

\begin{remark}[Common ordering]
For a fixed permutation matrix $\Pi$, the conjugate class $\Pi^{\mathsf T}\TAS\Pi$ is a semigroup, but factors based on different orderings cannot generally be mixed. For example, let
\[
 A=\begin{bmatrix}
 1&0&0\\
 1&0&0\\
 0&1/2&1/2
 \end{bmatrix}\in\TAS,
\]
\[
 B=\begin{bmatrix}
 1/2&1/2&0\\
 0&0&1\\
 0&0&1
 \end{bmatrix}.
\]
The matrix $B$ is almost Sarymsakov under the ordering $(2,3,1)$. Nevertheless,
\[
 AB=\begin{bmatrix}
 1/2&1/2&0\\
 1/2&1/2&0\\
 0&0&1
 \end{bmatrix}
\]
is not SIA. Thus the core dimension may vary only along one common nested ordering; the semigroup property does not extend to arbitrary mixtures of differently ordered classes.
\end{remark}

\begin{remark}\label{rem:cxpm}
If in~(\ref{eq:E}), $R\in\Tzero$, Lemma~\ref{lem:support} implies $\Tzero=\E(R)=\U(R)\subsetneq\TAS$ for $n\geq 3$; otherwise,
consider, for example, $R$ in~(\ref{eq:RP}) and $\tilde{R}$ in the following:
\[
 \tilde{R}=\begin{bmatrix}
 1&0&0\\
 1&0&0\\
 0&1/2&1/2
 \end{bmatrix}.
\]
Both matrices are SIA and GE, but not GR; in addition, $R^2$ and $\tilde{R}^2$ are in $\Tzero$. More importantly, both matrices
are in $\TAS(2)$. Comparing $R$ and $\tilde{R}$ suggests $\Tzero\subsetneq\E(R)=\U(R)\subsetneq\TAS$. Since $\TAS$ is an SIA semigroup, these counterexamples show that $\U(R)$ is not the largest SIA semigroup containing $\E(R)$.
\end{remark}

\section{Relations Among the Classes}
Table~\ref{tab:relations} summarizes the structural roles of the classes and their relation to the classical Sarymsakov class $\Tzero$. Here ``incomparable'' means that neither inclusion holds. The class $\TNE$ isolates multiplication compatibility of nested leading-block expansions, but it is not an SIA class. The accessible subclass $\TAS$ adds row-allowable entry into the selected core and is the variable-core consensus semigroup. The exact fixed-core horizon is $h_{n,r}$ in \eqref{eq:hnr}, whereas \eqref{eq:TAShorizon} is a dimension-only upper bound and is not claimed to be minimal except at $n=3$.

\begin{table}[!t]
\caption{Structural Roles and Relations to $\Tzero$}
\label{tab:relations}
\centering
\footnotesize
\setlength{\tabcolsep}{3pt}
\renewcommand{\arraystretch}{1.10}
\begin{tabular}{@{}p{0.17\columnwidth}p{0.32\columnwidth}p{0.43\columnwidth}@{}}
\toprule
Class & Structural basis & Relation and guarantee \\
\midrule
$\U(R)$ & powers of a prescribed $R$ & contains $\Tzero$; all members are SIA \\
$\TB(1)$ & fixed core with access & proper subset of $\Tzero$ for $n\ge2$ \\
$\TB(r)$  & fixed core with access & incomparable with $\Tzero$; SIA \\
$2\le r<n$ & &\\
$\TB(n)$ & full order & equals $\Tzero$ \\
$\TNE$ & nested GR/GE only & properly contains $\Tzero$ for $n\ge2$; not SIA in general \\
$\TAS$ & nested GR/GE with access & contains $\Tzero$, strictly for $n\ge3$; consensus semigroup \\
\bottomrule
\end{tabular}
\end{table}

For $r=1$, every row of a matrix in $\TB(1)$ has a positive first-column entry, so the matrix is scrambling; strict inclusion for $n\ge2$ follows from the rank-one matrix $\one e_2^{\mathsf T}$. For $r=n$, the lower-left condition is vacuous and $\TB(n)=\Tzero$. For every $2\le r<n$, the rank-one matrix $\one e_{r+1}^{\mathsf T}$ belongs to $\Tzero\setminus\TB(r)$, whereas the matrix whose first $r$ rows equal $e_1^{\mathsf T}$ and whose remaining rows equal $(e_2+e_i)^{\mathsf T}/2$ belongs to $\TB(r)\setminus\Tzero$.

Because $\TNE(n)=\Tzero$, the classical class is contained in $\TNE$. For $n\ge2$, the identity matrix lies in $\TNE(1)\setminus\Tzero$, so the inclusion is strict, but this enlargement has no consensus implication. Theorem~\ref{thm:TASsemigroup} gives the dimension-qualified relation for the accessible class $\TAS$. Thus $\TAS$ is not merely the set-theoretic union of the fixed-core classes: higher-order GE supplies multiplication compatibility, while row-allowability supplies access to the mixing core.

\section{Conclusion}
This note has developed three support-based constructions of
multiplication-closed stochastic-matrix families. The first replaces exact matching with prescribed power-support patterns
by support inclusion, yielding a finite power-generated semigroup that can enlarge the earlier equal-pattern
construction. The second uses a fixed ordered mixing core and row-allowable access to obtain an exact scrambling horizon and
uniform convergence guarantees for compact switched families.
The third and principal contribution is the almost Sarymsakov class, which permits the effective core dimension to vary among
factors while preserving multiplication closure and consensus.

The analysis also clarifies the roles of the structural assumptions. Nested expansion conditions provide compatibility
between different core dimensions, row-allowability prevents states outside the core from forming a closed group, and a
common ordering ensures that the leading-block conditions remain consistent under multiplication. The examples establish
strict class relations, demonstrate sharpness in low dimensions, and show that these assumptions cannot generally be removed.
Future work includes tighter scrambling-horizon estimates, efficient level-computation algorithms, and extensions allowing
multiple compatible core orderings.

%\section*{Acknowledgment} The author acknowledges the use of OpenAI ChatGPT (GPT-5.6 Pro) to assist with language editing throughout the manuscript, including the refinement of wording, grammar, and sentence-level clarity. All AI-assisted revisions were independently reviewed and verified by the author. The mathematical results, proofs, examples, interpretations, and conclusions were checked by the author, who assumes full responsibility for the content of the paper.

\bibliographystyle{IEEEtran}
\bibliography{paper_A_sarymsakov_TAC_almost_semigroup_emphasis}

\end{document}